\documentclass[a4paper,11pt]{amsart}
\usepackage[utf8x]{inputenc}

\usepackage{amssymb,amsfonts}
\usepackage{amsmath,amsthm}
\usepackage[all,arc, cmtip]{xy}
\usepackage{enumerate}
\usepackage{paralist}
\usepackage{booktabs} 
\usepackage{array} 
\usepackage{mathrsfs}
\usepackage{url}
\usepackage{color}
\usepackage{todonotes} 
\usepackage{hyperref} 
\usepackage{mathtools}
\usepackage[mathscr]{euscript}
\usepackage[nameinlink,capitalise,noabbrev]{cleveref}

\usepackage{tikz-cd}

\newcommand{\mc}{\mathcal}

\newcommand{\sss}{\scriptscriptstyle}

\newcommand{\binr}[2]{\ar@<.5ex>[r]^{#1} \ar@<-.5ex>[r]_{#2}}
\newcommand{\bind}[2]{\ar@<.5ex>[d]^{#1} \ar@<-.5ex>[d]_{#2}}
\newdir{ >}{{}*!/-5pt/\dir{>}}
\newdir^{ (}{{}*!/-5pt/@^{(}}

\def\ra{\rightarrow}

\def\cross{\mathrm{cr}}

\def\QQ{{\mathbb{Q}}}
\def\ZZ{{\mathbb{Z}}}

\def\PP{{\mathbb{P}}}

\def\cB{{\mathcal B}}
\def\cM{{\mathcal M}}

\def\cO{{\mathcal O}}
\def\cP{{\mathcal P}}

\DeclareMathOperator\Aut{Aut}

\newtheorem{thm}{Theorem}[section]
\newtheorem*{thm*}{Theorem}

\newtheorem{cor}[thm]{Corollary}

\newtheorem{lem}[thm]{Lemma}

\theoremstyle{definition}

\newtheorem{defn}[thm]{Definition}
\newtheorem*{defn*}{Definition}

\newtheorem{rem}[thm]{Remark}

\theoremstyle{remark}

\makeatletter
\let\c@equation\c@thm
\makeatother
\numberwithin{equation}{section}
\title[Formulae for Exterior power operations]
{Two Formulae for Exterior power operations on higher $K$-groups}

\author{Tom Harris}
\address{University Printing House\\ Shaftesbury Road\\ Cambridge CB2 8BS\\ United Kingdom}
\email{tharris@cambridge.org}

\author{Bernhard K\"{o}ck}
\address{Mathematical Sciences\\ University of Southampton\\ Highfield\\ Southampton SO17 1BJ\\ United Kingdom}
\email{b.koeck@soton.ac.uk}

\date{\today}

\begin{document}

\begin{abstract}
Exterior power operations on the higher $K$-groups of a quasi-compact scheme have recently been constructed by Taelman and the authors by purely algebraic means. In this paper, we prove two formulae that help to compute these operations. The first is a formula for exterior powers of external products. The second is a formula for exterior powers of $n$-cubes, i.e., of acyclic binary multi-complexes supported on~$[0,1]^n$. These formulae provide evidence for the expectation that our exterior power operations agree with those defined by Hiller.
\end{abstract}



\maketitle

\section*{Introduction}

Let $X$ be a quasi-compact scheme. Following Grayson's algebraic description \cite{Grayson2012} of higher $K$-groups in terms of explicit generators and relations, Taelman and the authors of this paper have, in \cite{HKT17}, constructed an exterior power operation $\lambda^r$ on $K_n(X)$ for every $r \ge 1$ and $n \ge 0$ by assigning an explicit element $\lambda^r(\PP)$ to each of Grayson's generators $\PP$. While our construction is purely algebraic (i.e., unlike all previous constructions does not use any homotopy theoretic methods) and $\lambda^r(\PP)$ can in principle be written down, this is normally too combinatorially complex to be done by hand. The purpose of this paper is to augment our construction by two formulae that help to compute~$\lambda^r$.

Our first formula (\cref{thm: ExternalProducts}) says that for all $x \in K_m(X)$ and all $y \in K_n(X)$ we have
\[\lambda^r(x \smallsmile y) = (-1)^{r-1}\; r\; \lambda^r(x) \smallsmile \lambda^r(y) \quad \textrm{ in } \quad K_{m+n}(X).\]
Here, $\smile$ denotes the external product between higher $K$-groups defined in \cref{def: ExternalProducts} below. The main ingredient in the proof is the characteristic-free Cauchy decomposition of the $r^\mathrm{th}$ exterior power of a tensor product as constructed by Akin, Buchsbaum and Weyman in \cite{ABW}. \cref{cor: AdamsOperations} says that the corresponding Adams operations are multiplicative with respect to~$\smile$, as expected by classical formulae.

Our second formula (\cref{thm: Exterior Powres of Cubes}) computes the operation~$\lambda^r$ when applied to certain generators of $K_n(X)$, namely to those which are given by a locally-free $\cO_X$-module $P$ of finite rank together with $n$ pairwise-commuting automorphisms $A_1, \ldots, A_n$, and which we denote $[P;A_1, \ldots, A_n]$. For the reader familiar with the notion of binary multi-complexes (see beginning of \cref{sec: ExternalProducts} for details), we explain already now that $[P;A_1, \ldots, A_n]$ is the $n$-dimensional acyclic binary multi-complex whose entry at every place $(j_1, \ldots, j_n) \in \{0,1\}^n$ is $P$ and is equal to $0$ everywhere else and whose non-trivial edges in direction $i \in \{1, \ldots, n\}$ are all equal to the binary isomorphism $\xymatrix{P \binr{A_i}{1}& P}$. Then, for $r=2$, the element $\lambda^r[P; A_1, \ldots, A_n]$ is equal to
\[[\Lambda^2(P); \Lambda^2(A_1), \ldots, \Lambda^2(A_n)] - [P\otimes P; A_1 \otimes 1, A_2 \otimes A_2, \ldots, A_n \otimes A_n],\]
and, for $r=3$, it is equal to
\begin{align*}
&[\Lambda^3(P); \Lambda^3(A_1), \ldots, \Lambda^3(A_n)]\\
& - [\Lambda^2(P) \otimes P; \Lambda^2(A_1) \otimes 1, \Lambda^2(A_2) \otimes A_2, \ldots, \Lambda^2(A_n) \otimes A_n]\\
& - [P \otimes \Lambda^2(P); A_1 \otimes 1, A_2 \otimes \Lambda^2(A_2), \ldots, A_n \otimes \Lambda^2(A_n)]\\
& + [P \otimes P \otimes P; A_1 \otimes 1 \otimes 1, A_2 \otimes A_2 \otimes A_2, \ldots, A_n \otimes A_n \otimes A_n];
\end{align*}
for arbitrary $r$, our formula describes $\lambda^r[P; A_1, \ldots, A_n]$ in terms of the cross-effects of $\Lambda^r$, see \cref{thm: Exterior Powres of Cubes}. Its proof is based on the explicit description of the $r^\mathrm{th}$ exterior power of an ordinary chain complex supported on $[0,1]$ as given in Lemma~2.2 of \cite{Koszul}. It proceeds by induction on $n$ and uses a generalised version of the fact that internal products of elements in $K_n(X)$ vanish if $n \ge 1$, see \cref{lem: tensorproducts}.

In \cref{cor: Hiller}, we derive from \cref{thm: Exterior Powres of Cubes} that our $\lambda^r$ agrees with the operation $\lambda^r$ defined by Hiller \cite{Hiller} on the $K_1$-group of any commutative ring~$R$.  Assuming that our external products agree with those used by Hiller, we moreover conclude using \cref{cor: AdamsOperations} that the corresponding Adams operations agree on external products of $K_1$-elements. We of course expect that not only the Adams operations, but already the exterior power operations agree, and not only on external products of elements of $K_1(X)$ but on arbitrary elements of $K_n(X)$.

In his recent PhD thesis \cite{Grech}, Grech introduces modified Milnor $K$-groups $\widetilde{K}^\mathrm{M}_n(R)$ for any ring $R$ together with canonical maps $\widetilde{K}^\mathrm{M}_n(R) \rightarrow K_n(R)$. \cref{thm: Exterior Powres of Cubes} implies that the image of these maps is invariant under $\lambda^r$, see \cref{cor: MilnorInvariant}. This raises the question: do there exist exterior power operations on $\widetilde{K}^\mathrm{M}_n(R)$ as well?

\medskip
{\bf Acknowledgments.} The authors thank Daniel Grech, Marco Schlichting and Ferdinando Zanchetta for stimulating discussions and for their interest in our results.

\section{Exterior Powers of External Products}\label{sec: ExternalProducts}

In this section, we first define (\cref{def: ExternalProducts}) explicit external products
\[\xymatrix{K_m(X) \times K_n(X) \ar[r]^{\hspace*{1.5em}\smile} & K_{m+n}(X)}\]
on the higher $K$-groups of a quasi-compact scheme $X$, using Grayson's algebraic presentation of higher $K$-groups. We then state and prove a formula for the $r^\mathrm{th}$ exterior power operation applied to external products, see \cref{thm: ExternalProducts}. It implies (\cref{cor: AdamsOperations}) that the corresponding Adams operations are multiplicative with respect to external products, as expected by classical formulae.

\medskip

\noindent Let $X$ be a quasi-compact scheme and let $\cP(X)$ denote the exact category of locally free $\cO_X$-modules of finite rank. We recall an {\em acyclic binary complex} $\PP=(P_*, d, \tilde{d})$ in~$\cP(X)$ is a graded object $P_*$ in $\cP(X)$ supported on a finite subset of $[0,\infty)$ together with two degree $-1$ maps $d, \tilde{d}:P_* \rightarrow P_*$ such that both $(P_\ast, d)$ and $(P_*,\tilde{d})$ are acyclic chain complexes. If $d=\tilde{d}$,  the complex~$\PP$ is said to be {\em diagonal}.

A {\em morphism between acyclic binary complexes} $\PP$ and $\QQ$ is a degree 0 map between the underlying graded objects that is a chain map with respect to both differentials. The obvious definition of short exact sequences turns the category of acyclic binary complexes into an exact category. This category will be denoted $B^\mathrm{q}_\mathrm{b}(\cP(X))$ (as in \cite{HKT17}) or simply $\cB(X)$.

Iterating this construction $n$ times we obtain the category
\[\cB^n(X):= B_\mathrm{b}^\mathrm{q}\left(\cB^{n-1}(X)\right)\]
of {\em acyclic binary multicomplexes of dimension $n$}. Explicitly, an object of~$\cB^n(X)$ is a  $\ZZ_{\ge 0}^n$-graded object in $\cP(X)$ equipped with
\textit{two} (acyclic) differentials, denoted $d^i$ and $\tilde{d}^i$, in each direction
 $1 \le i \le n$, such that $d^id^j = d^jd^i$, $d^i\tilde{d}^j = \tilde{d}^jd^i$, $\tilde{d}^id^j = d^j\tilde{d}^i$ and $\tilde{d}^i\tilde{d}^j = \tilde{d}^j\tilde{d}^i$ whenever $i \neq j$. If $d^i = \tilde{d}^i$ for at least one~$i$, the multicomplex is said to be {\em diagonal}. We will follow the convention that the $i^\mathrm{th}$ coordinate of $\ZZ_{\ge 0}^n$ corresponds to the $i^{th}$ application of the operation $B^\mathrm{q}_\mathrm{b}$. When depicting objects of $\cB^2(X)$, differentials in the first and second direction will be displayed horizontally and vertically, respectively.

\begin{defn}\label{def:GraysonKn}
Let $n\ge 0$. The quotient of the Grothendieck group $K_0(\cB^n(X))$ obtained by declaring the classes of diagonal acyclic binary multicomplexes complexes to be zero is called the {\em $n^\textrm{th}$ $K$-group of $X$} and denoted~$K_n(X)$.
\end{defn}

Grayson proves in \cite[Corollary 7.2]{Grayson2012} that $K_n(X)$ is naturally isomorphic to Quillen's $n^\textrm{th}$ $K$-group of $X$. This justifies our notation. Note that Grayson uses complexes supported on $(-\infty, \infty)^n$; by \cite[Proposition~1.4]{HKT17}, defining $K_n(X)$ with complexes supported on~$[0,\infty)^n$ as above yields the same $K_n$-group.

\begin{defn}\label{def: ExternalProducts}
Given binary multicomplexes $\PP \in \cB^m (X)$ and $\QQ \in \cB^n (X)$, we form their {\em external tensor product} $\PP \boxtimes \QQ \in \cB^{m+n}(X)$ as follows: if $n=0$, i.e., if $\QQ$ consists of a single object $Q$, the multicomplex $\PP \boxtimes \QQ$ is obtained from $\PP$ by tensoring every object in $\PP$ with $Q$ and every differential with~$1_\QQ$, see also \cite[Section~7]{HKT17}. If $n\ge 1$, we write $\QQ=(Q_\ast, d^n_\QQ, \tilde{d}^n_\QQ)$ with $Q_i \in \cB^{n-1}(X)$ and recursively define:
\[\PP \boxtimes \QQ: \quad \xymatrix{ \cdots \binr{}{} &\PP \boxtimes Q_2 \binr{\sss 1_\PP \otimes d^n_\QQ}{\sss 1_\PP \otimes \tilde{d}^n_\QQ} &\PP \boxtimes Q_1 \binr{\sss 1_\PP \otimes d^n_\QQ}{\sss 1_\PP \otimes \tilde{d}^n_\QQ} & \PP \boxtimes Q_0}, \]
where the displayed differentials point in the $(m+n)^\mathrm{th}$ direction.
\end{defn}

For example, if $m=n=1$, the product $\PP \boxtimes \QQ$ is given by the following diagram:
 \[  \xymatrix{
  & \vdots \bind{}{}& \vdots \bind{}{}\\
  \dots \binr{}{} & \hspace*{1em} P_1 \otimes Q_1 \hspace*{1em} \binr{\sss \hspace*{0.3em} d_\PP \otimes 1}{\sss \hspace*{0.3em} \tilde{d}_\PP \otimes 1} \bind{\sss 1 \otimes {d}_\QQ}{\sss 1 \otimes \tilde{d}_\QQ} &\hspace*{1em} P_0 \otimes Q_1 \bind{\sss 1 \otimes {d}_\QQ}{\sss 1 \otimes \tilde{d}_\QQ}\\
  \dots \binr{}{} & \hspace*{1em} P_1 \otimes Q_0 \hspace*{1em} \binr{\sss \hspace*{0.3em} d_\PP \otimes 1}{\sss \hspace*{0.3em}\tilde{d}_\PP \otimes 1} & \hspace*{1em} P_0 \otimes Q_0;
  }
 \]

\begin{lem}
The product $(\PP, \QQ) \mapsto \PP \boxtimes \QQ$ induces a well-defined graded ring structure on $K_*(X) = \bigoplus_{n \ge 0} \;K_n(X)$.
\end{lem}

\noindent We use the symbol $\smallsmile$ to denote this ring structure.

\begin{proof}
Straightforward; for details see \cite[Lemma~2.26]{HarPhD}.
\end{proof}

\begin{rem}

Permuting the directions (while not changing the role of top and bottom differential), defines an action of the symmetric group $\Sigma_n$ on $\cB^n(X)$ and hence on $K_n(X)$. We conjecture that any  $\sigma \in \Sigma_n$ acts on $K_n(X)$ by multiplication by $\mathrm{sgn}(\sigma)$. If true, this implies that the graded ring $K_*(X)$ is anti-commutative.

Moreover, via Grayson's identification of $K_n(X)$ with Quillen's $K_n$-group, we expect our ring structure $\smallsmile$ to agree with the ring structure defined in \cite[Definition~IV.6.6]{K-book}. Assuming this, \cite[Theorem IV.1.10]{K-book} then implies as well that $K_*(X)$ is anti-commu\-ta\-tive. Similarly, from \cite[III \S 7]{K-book} we obtain that the Steinberg relation holds: given a field $F$ and $a \in F \backslash \{0, 1\}$, we have
\[[a] \smallsmile [1-a] = 0 \quad \textrm{ in } \quad K_2(F),\]
where $[a]$ denotes the class in $K_1(F)$ of the acyclic binary complex
\[\xymatrix{
\dots \binr{0}{0} & F \binr{a}{1} & F.
}\]
We do not know of a proof of these statements that is purely algebraic and that, more specifically, is intrinsic to the binary complex context. In \cite[Section~5.3]{Grech}, Grech provides such a proof for the Steinberg relation by using the homotopy theorem which however has not as yet been proved algebraically.
\end{rem}

In \cite{HKT17}, we have defined exterior power operations $\lambda^r$, $r \ge 1$, on $K_n(X)$, $n \ge 0$, and shown that these operations make $K_*(X)$ into a (special) $\lambda$-ring. There, the product on $\bigoplus_{n \ge 1} K_n(X)$ is zero; that product may also be viewed as the one induced by the simplicial tensor product, see \cite[Section~5]{HKT17}, and should therefore perhaps be called the {\em internal product}. The following theorem provides a formula for exterior powers of {\em external} products.

\begin{thm}\label{thm: ExternalProducts}
Let $m, n, r >0$, $x \in K_m(X)$ and $y \in K_n(X)$. Then we have
\[\lambda^r(x \smallsmile y) = (-1)^{r-1}\; r\; \lambda^r(x) \smallsmile \lambda^r(y) \quad \textrm{ in } \quad K_{m+n}(X).\]
\end{thm}

\begin{proof}
We may assume that $x$ and $y$ are the classes of $m$- and $n$-dimensional acyclic binary multi-complexes $\PP$ and $\QQ$, respectively. According to {\em loc.\ cit.},  $\lambda^r(y)$ is defined as the class of  $\Lambda^r_n(\QQ):= N^n \Lambda^r \Gamma^n(\QQ) \in \cB^n(X)$, where we write $\Gamma^n$ and $N^n$ for the $n$-dimensional versions of the Dold--Kan correspondence functors $\Gamma$ and $N$ ({\em cf.}\ Remark~3.6 in {\em loc.\ cit.}). Hence, for every object~$P$ in $\mc{P}(X)$, we have
\[\Lambda^r_n(P \otimes \QQ) = N^n \Lambda^r \Gamma^n(P \otimes \QQ) = N^n\Lambda^r(P \otimes \Gamma^n(\QQ)).\]
By \cite[Theorem~III.2.4]{ABW}, the `$n$-simplicial binary object' $\Lambda^r (P \otimes \Gamma^n(\QQ))$ has a natural filtration with the following property: when $\mu$ runs through the set of partitions of weight~$r$, the tensor product $L_\mu(P) \otimes K_\mu(\Gamma^n(\QQ))$ runs through the sequence of successive quotients of this filtration; here, $L_\mu$ and $K_\mu$ denote the Schur and coSchur functors associated with the partition~$\mu$ \cite[pp.~219--220]{ABW}. As $L_\mu$ and $K_\mu$ are endo-functors of $\cP(X)$ \cite[Theorem~II.2.16]{ABW},  all objects in the filtration steps are in $\cP(X)$ as well. Since $N^n$ is exact and commutes with $L_\mu(P) \otimes -$, the object $\Lambda^r_n(P \otimes \QQ)$ therefore has a natural filtration by subobjects belonging to $\cB^n(X)$ whose successive quotients are $L_\mu(P) \otimes K_{\mu,n}(\QQ)$, where $K_{\mu, n}$ denotes the endo-functor $N^n K_\mu \Gamma^n$ of $\cB^n(X)$, see also~\cite[Section 4]{HKT17}. Repeating the above argument, with the roles of the first and second factors interchanged and with $\Lambda^r$ replaced with $\Lambda^r_n$, we conclude that the object
\[\Lambda_{m+n}^r(\PP \boxtimes \QQ) = N^m \Lambda_n^r\Gamma^m(\PP \boxtimes \QQ) = N^m \Lambda^r_n (\Gamma^m(\PP) \boxtimes \QQ)\]
has a filtration by subobjects belonging to $\cB^{m+n}(X)$ whose successive quotients are $L_{\mu,m}(\PP) \boxtimes K_{\mu,n}(\QQ)$. Using the universal form of the Pieri formula \cite[Section~3, Thoerem~(3)(a)]{AB85}, one shows as in \cite[Proposition~2.1]{Classgroups} that
\[[L_{\mu,m}(\PP)] = s_\mu\left([\Lambda^1_m(\PP)], \ldots, [\Lambda_m^r(\PP)]\right) \quad \textrm{ in } \quad K_0(\cB^m(X))\]
and
\[[K_{\mu,n}(\QQ)] = s_{\tilde{\mu}}\left([\Lambda^1_n(\QQ)], \ldots, [\Lambda^r_n(\QQ)]\right) \quad \textrm{ in } \quad K_0(\cB^n(X));\]
here, $s_\mu$ and $s_{\tilde{\mu}}$ denote the Schur polynomials associated with ${\mu}$ and with the transposed partition $\tilde{\mu}$, respectively, see \cite[I Section~3]{McD79}; note that the multiplication on the right-hand sides of these formulae is given by the simplicial tensor products
\[ \otimes_{\Delta,m}: \cB^m(X) \times \cB^m(X) \rightarrow \cB^m(X) \;\; \textrm{and} \;\; \otimes_{\Delta,n}: \cB^n(X) \times \cB^n(X) \rightarrow \cB^n(X)\]
constructed in \cite[Section~5]{HKT17}. By \cite[p.~35, formula (4.3')]{McD79} we furthermore have
\[\sum_{|\mu|=r} s_\mu(X_1, \ldots, X_r) s_{\tilde{\mu}}(Y_1, \ldots, Y_r) = P_r (X_1, \ldots, X_r, Y_1, \ldots, Y_r)\]
in $\mathbb{Z}[X_1, \ldots, X_r, Y_1, \ldots, Y_r]$, where $P_r (X_1, \ldots, X_r, Y_1, \ldots, Y_r)$ denotes the integral polynomial used in the formulation of the $\lambda$-ring axiom for products, see \cite[Equation~I(1.3)]{FultonLang}. We therefore obtain
\[[\Lambda^r_{m+n}(\PP \boxtimes \QQ)] = P_r\left([\Lambda_m^1(\PP)], \ldots, [\Lambda_m^r(\PP)], [\Lambda_n^1(\QQ)], \ldots [\Lambda^r_n(\QQ)]\right)\]
in $K_0\left(\cB^{m+n}(X) \right)$; note that the polynomial on the right-hand side is evaluated using the simplicial tensor products $\otimes_{\Delta,m}$, $\otimes_{\Delta,n}$ and the external tensor product $\boxtimes$ appropriately. Passing to $K_{m+n}(X)$ we finally obtain
\[[\Lambda^r_{m+n}(\PP \boxtimes \QQ)] = P_r\left(0, \ldots, 0, [\Lambda_m^r(\PP)], 0, \ldots, 0, [\Lambda_n^r(\QQ)]\right)\]
in $K_{m+n}(X)$, as all simplicial tensor products vanish by \cite[Proposition~5.11]{HKT17} and as the polynomial $P_r$ is homogeneous of weighted degree $r$ in both $X_1, \ldots, X_r$ and $Y_1, \ldots, Y_r$. Now Theorem~\ref{thm: ExternalProducts} follows from the formula
\[P_r(0, \ldots, 0, X_r, 0, \ldots, 0, Y_r) = (-1)^{r-1} \, r \, X_r Y_r  \textrm{ in } \mathbb{Z}[X_1, \ldots, X_r, Y_1, \ldots, Y_r],\]
which for instance is implied by the well-known formulae
\[N_r(0, \ldots, 0, X_r) = (-1)^{r-1} \, r \, X_r\]
and
\begin{align*}
N_r(X_1, \ldots, X_r) \cdot & N_r(Y_1, \ldots, Y_r)\\&
= N_r (P_1(X_1, Y_1), \ldots, P_r(X_1, \ldots, X_r, Y_1, \ldots, Y_r))
\end{align*}
for the $r^\textrm{th}$ Newton polynomial $N_r$.
\end{proof}

\begin{cor}\label{cor: AdamsOperations}
For $x,y$ as in Theorem~\ref{thm: ExternalProducts} and for the $r^\textrm{th}$ Adams operation~$\psi^r$ we have
\[\psi^r(x \smallsmile y) = \psi^r(x) \smallsmile \psi^r(y) \quad \textrm{ in } \quad K_{m+n}(X).\]
\end{cor}

\begin{proof}
This immediately follows from the previous theorem and the formula
\[\psi^r(x) = (-1)^{r-1}\, r \, \lambda^r(x)\]
for all $x \in K_n(X)$ and all $n > 0$ (see also end of proof of previous theorem).
\end{proof}

\section{Exterior Powers of $n$-Cubes}

The main object of this section is to give a formula (see \cref{thm: Exterior Powres of Cubes} and \cref{rem: ArbitraryCubes}) for the $r^\mathrm{th}$ exterior power operation applied to an $n$-cube, i.e., to a multi-complex in $\cB^n(X)$ that is supported only on $[0,1]^n$. In \cref{cor: Hiller}, we deduce that the exterior power operations defined in \cite{HKT17} agree with those defined by Hiller \cite{Hiller} on the $K_1$-group of a commutative ring $R$. In \cref{cor: MilnorInvariant}, we moreover derive that the image of the newly defined modified Milnor $K$-group $\widetilde{K}^\mathrm{M}_n(R)$ in $K_n(R)$ is invariant under exterior power operations. We begin with some notation.

\medskip

Let $F: \cP \ra \cM$ be a functor from an additive category $\cP$ to an idempotent-complete additive category $\cM$ such that $F(0) =0$. The {\em $i^\textrm{th}$ cross-effect of~$F$} is  the functor $\cross_i F: \cP^i \ra \cM$ defined inductively as follows: $\cross_1 F := F$, and, for objects $P_1, \ldots, P_i \in \cP$, the object $\cross_i F(P_1, \ldots, P_i) \in \cM$ is the canonical direct-sum complement of
\[\cross_{n-1}F(P_1, P_3, \ldots, P_i) \oplus \cross_{n-1}F(P_2, P_3, \ldots, P_i)\]
in $\cross_{n-1}F(P_1 \oplus P_2, P_3, \ldots, P_i)$, see for instance \cite[Section~9]{EM} or \cite[Section~1]{Koszul} for more details. For example, we have
\[\cross_2\Lambda^3(P_1, P_2) = \Lambda^2(P_1) \otimes P_2 \oplus P_1 \otimes \Lambda^2(P_2)\]
for the third exterior power functor $\Lambda^3$ and we have
\[\cross_r\Lambda^r(P_1, \ldots, P_r) = P_1 \otimes\cdots \otimes P_r\]
for all $r \ge 1$.

Let $\Aut(\cP)$ denote the category of automorphisms of $\cP$: its objects are pairs~$(P;A)$ consisting of an object $P \in \cP$ and an automorphism~$A$ of~$P$, and the morphisms from $(P;A)$ to $(P';A')$ are those morphisms $B$ from~$P$ to $P'$ that satisfy $BA=A'B$. By iterating this definition we obtain the categories $\Aut^n(\cP)$, $n \ge 1$. Thus, an object of $\Aut^n(\cP)$ is a tuple $(P; A_1, \ldots, A_n)$ consisting of an object $P$ of $\cP$ and of $n$ pairwise commuting automorphisms $A_1, \ldots, A_n$ of $P$. If $\cP$ is an exact category, then sending $(P,A)$ to the (bounded acyclic) binary complex
\[\xymatrix{
\cdots \binr{0}{0} & P \binr{A}{1} & P
}\]
defines a functor $\Aut(\cP) \ra B^\mathrm{q}_\mathrm{b}(\cP)$. By iterating we obtain functors
\[\Aut^n(\cP) \ra (B^\mathrm{q}_\mathrm{b})^n(\cP), \qquad n \ge 1.\]
We denote the image of $(P;A_1, \ldots, A_n)$ in $(B^\mathrm{q}_\mathrm{b})^n(\cP)$ by $[P;A_1, \ldots, A_n]$; i.e., $[P, A_1, \ldots, A_n]$ is the $n$-cube with every vertex equal to $P$ and and with every binary morphism in direction $i \in \{1, \ldots, n\}$ equal to
\[\xymatrix{P \binr{A_i}{1}& P}.\]

\begin{thm}\label{thm: Exterior Powres of Cubes}
Let $X$ be a quasi-compact scheme, let $r, n \ge 1$ and let $(P; A_1, \ldots, A_n) \in \Aut^n(\cP(X))$. Then, in $K_n(X)$, we have:
\[
\lambda^r[P;A_1, \ldots, A_n]
= \sum_{i=1}^r (-1)^{i-1}[\cross_i\Lambda^r(P, \ldots, P); B_1, \ldots, B_n],
\]
where $B_1 = \cross_i\Lambda^r(A_1, 1, \ldots, 1)$ and $B_j = \cross_i\Lambda^r(A_j, \ldots, A_j)$ for $j=2, \ldots, n$.
\end{thm}

For example, if $r=2$ or $r=3$, the right-hand element specialises to the expressions given in the introduction. Another interesting example is when the $\cO_X$-module $P$ is invertible and hence $A_1, \ldots, A_n \in \cO_X(X)^\times$; then, using the multi-linearity relations \cite[Section~5.1]{Grech}, we obtain
\begin{align*}
\lambda^r[P; A_1, \ldots, A_n] &= (-1)^{r-1}[P^{\otimes r}; A_1, A_2^r, \ldots, A_n^r]\\
& = (-1)^{r-1} r^{n-1} [P^{\otimes r}; A_1, \ldots, A_n];
\end{align*}
for instance, if $P$ is free of rank~$1$, then $\lambda^r$ acts on $[P; A_1, \ldots, A_n]$ by multiplication by $(-1)^{r-1} r^{n-1}$; this in turn implies that the  $r^\mathrm{th}$ Adams operation $\psi^r= (-1)^{r-1} r \lambda^r$ acts on  $[P; A_1, \ldots, A_n]$ by multiplication by $r^n$ which is in agreement with the classical result \cite[Example~5.9.1]{K-book} for the case when $X=\mathrm{Spec}(k)$, $k$ a field.

\begin{rem}\label
The reader may wonder why the first direction plays a distinguished role in \cref{thm: Exterior Powres of Cubes} above. We will see in the proof below that this is a result of choosing a particular ordering of the directions. As explained in \cite[Remark~3.6]{HKT17}, the inductive definition of exterior power operations on $K_n(X)$ given in \cite[Definitions~3.3 and~4.4]{HKT17} does not depend on the chosen ordering of directions. The same is true for the right-hand side of the formula in \cref{thm: Exterior Powres of Cubes}. For instance when $n=2$, we can see this directly (i.e., without using \cref{thm: Exterior Powres of Cubes}) as follows, writing $F_i$ for $\cross_i \Lambda^r$ and using the multi-linearity relation \cite[Section~5.1]{Grech} and the intrinsic symmetry of cross effects \cite[Theorem~9.3]{EM}:
\begin{align*}
[F_i &(P, \ldots, P); F_i(A_1, 1, \ldots, 1), F_i(A_2, \ldots, A_2)]\\
&=  [F_i(P, \ldots, P);F_i(A_1, 1, \ldots, 1), F_i(A_2, 1, \ldots, 1)]\; +\\
&\hspace*{1.4em} [F_i(P, \ldots, P);F_i(A_1, 1, \ldots, 1), F_i(1, A_2, 1, \ldots, 1)]+ \cdots + \\
& \hspace*{1.4em}[F_i(P, \ldots, P);F_i(A_1, 1, \ldots, 1), F_i(1, \ldots, 1, A_2)]\\
&= [F_i(P, \ldots, P); F_i(A_1, 1, \ldots, 1), F_i(A_2, 1, \ldots, 1)]\; +\\
&\hspace*{1.4em} [F_i(P, \ldots, P);F_i(1, A_1, 1, \ldots, 1), F_i(A_2, 1, \ldots, 1)]+ \ldots +\\
&\hspace*{1.4em} [F_i(P, \ldots, P);F_i(1, \ldots, 1, A_1), F_i (A_2, 1, \ldots, 1)]\\
&= [F_i(P, \ldots, P);F_i(A_1, \ldots, A_1), F_i(A_2, 1, \ldots, 1)].
\end{align*}
\end{rem}

\begin{proof}[Proof of \cref{thm: Exterior Powres of Cubes}]
  It is more intuitive to explain the main argument in the general context of the functor $F:\cP \ra \cM$ introduced at the beginning of this section.

  Let $\Gamma$ and $N$ denote the functors of the Dold--Kan correspondence, so the composition $NF\Gamma$ takes a non-negatively supported chain complex in $\cP$ first to a simplicial object in $\cP$, then to a simplicial object in $\cM$ and finally to a non-negatively supported chain complex in $\cM$. Furthermore, let $A: P \ra Q$ be a morphism in $\cP$ considered as a chain complex supported on $[0,1]$. Then the complex $NF\Gamma(\xymatrix{P \ar[r]^{A}& Q})$ in $\cM$ is isomorphic to the total complex of the double complex
  \[\xymatrix{
  \cdots \ar[r]& \cross_3 F(P,P,P) \ar[r] \ar[d]^{\cross_3 F(A, 1, 1)} & \cross_2 F(P,P) \ar[r] \ar[d]^{\cross_2 F(A,1)} & F(P) \ar[d]^{F(A)}\\
  \cdots \ar[r]& \cross_3 F(Q,P,P) \ar[r] & \cross_2 F(Q,P) \ar[r] & F(Q)
  }\]
  whose horizontal differentials are explicitly described in \cite[Lemma~2.2]{Koszul}. In fact, this statement is just a reformulation of that lemma. Moreover, if $(A_P, A_Q)$ is a morphism from the complex $\xymatrix{P \ar[r]^A &Q}$ to another complex $\xymatrix{P' \ar[r]^{A'}& Q'}$, then the induced morphism from $NF\Gamma(\xymatrix{P \ar[r]^A &Q})$ to $NF\Gamma(\xymatrix{P' \ar[r]^{A'} & Q'})$ is given, via the above isomorphism, by the morphisms $\cross_i F(A_P, \ldots, A_P)$, $i \ge 1$, and $\cross_i F(A_Q, A_P, \ldots, A_P)$, $i \ge 1$.

  By \cite[Definitions~3.3 and~4.4]{HKT17}, the $r^\mathrm{th}$ exterior power functor
  \[\Lambda_m^r: \cB^m(X) \ra \cB^m(X)\]
  for $m\ge 1$ is inductively defined by $\Lambda^r_m := N \Lambda^r_{m-1} \Gamma$ where $\Gamma$ and $N$ are the appropriately defined Dold--Kan correspondence functors. Similarly to the notation $[P; A_1, \ldots, A_n]$, for $m=1, \ldots, n$, we define the object
  \begin{equation}\label{object}
  [\Lambda^r_m[P; A_1, \ldots, A_m]; \Lambda^r_m(A_{m+1}), \ldots, \Lambda^r_m(A_n)] \in \cB^n(X)
  \end{equation}
  supported on $[0, \infty)^m \times [0,1]^{n-m}$ by considering $A_{m+1}, \ldots, A_n$ as automorphisms of $(P;A_1, \ldots, A_m)$ and then turning the automorphisms $\Lambda_m^r(A_{m+1})$, $\ldots, $ $\Lambda_m^r(A_n)$ of $\Lambda_m^r[P;A_1, \ldots, A_n]$ into binary isomorphisms in directions $m+1, \ldots, n$, respectively. Now, after filtering the double complex above vertically, the previous paragraph implies that the class of the element~(\ref{object}) in $K_n(X)$ is equal to
    \begin{equation}\label{element}
    \sum_{i=1}^r (-1)^{i-1}[\cross_i \Lambda_{m-1}^r ([P; A_1, \ldots A_{m-1}]); C_{m}, \ldots, C_n]
    \end{equation}
  where $C_{m} = \cross_i \Lambda_{m-1}^r (A_m, 1, \ldots, 1)$ and $C_j = \cross_i \Lambda_{m-1}^r(A_{j}, \ldots, A_j)$ for $j \in \{m+1, \ldots, n\}$.
  We'll prove below that, if $m \ge 2$, every summand in (\ref{element}) vanishes except the one for $i=1$. Then we obtain
  \begin{align*}
  \lambda^r&[P; A_1, \ldots, A_n]
  = \Lambda^r_n[P; A_1, \ldots, A_n]\\
  & = [\Lambda_{n-1}^r[P;A_1, \ldots, A_n]; \Lambda_{n-1}^r(A_n)]\\
  & = \ldots \\
  &= [\Lambda_1^r[P;A_1]; \Lambda_1^r(A_2), \ldots, \Lambda_1^r(A_n)]\\
  & = \sum_{i=1}^r (-1)^{i-1}[\cross_i\Lambda^r(P); B_1, \ldots, B_n],
  \end{align*}
  as was to be shown.

  It remains to show the above statement about the sum~(\ref{element}). An easy induction on $m$ yields the isomorphism
  \[\Lambda_m^r(\PP \oplus \QQ) \cong \Lambda^r_m(\PP)\oplus \bigoplus_{j=1}^{r-1} \left(\Lambda_m^j(\PP) \otimes_{\Delta,m} \Lambda_m^{r-j}(\QQ) \right) \oplus \Lambda_m^r(\QQ).\]
  for any objects $\PP, \QQ \in \cB^n(X)$ and $m\ge 0$; here, $\otimes_{\Delta,m}$ denotes the simplicial tensor product introduced in \cite[Section~5]{HKT17}. From this one easily derives that
  \[\cross_i \Lambda^r_{m}(\PP_1, \ldots, \PP_i) \cong \bigoplus_{j_1, \ldots, j_i \ge 1: \, j_1 + \cdots +j_i =r} \Lambda_{m}^{j_1}(\PP_1) \otimes_{\Delta, m} \cdots \otimes_{\Delta,m} \Lambda_{m}^{j_i}(\PP_i)\]
  for any $\PP_1, \ldots, \PP_i \in \cB^m(X)$. By \cite[Proposition~5.11]{HKT17}, the class of this element in $K_m(X)$ vanishes if $m\ge 1$ and $i \ge 2$. However, in order to see the above statement about the sum~(\ref{element}), we need the following more general version of  that proposition.
\end{proof}

\begin{lem}\label{lem: tensorproducts}
Let $n \ge m \ge 1$. Let $\PP$ and $ \QQ$ be objects in $\cB^m(X)$, let $A_{m+1}, \ldots, A_n \in \Aut(\PP)$ and let $B_{m+1}, \ldots, B_n \in \Aut(\QQ)$. Then we have:
\[[\PP \otimes_{\Delta,m} \QQ; A_{m+1} \otimes_{\Delta,m} B_{m+1}, \ldots, A_n \otimes_{\Delta,m} B_n] =0 \quad \textrm{ in } \quad K_n(X).\]
\end{lem}

\begin{proof} This follows by adapting the proof of \cite[Proposition~5.11]{HKT17} to the more general situation. The following lines explain the main steps.

We write $\PP$ as a binary complex of objects $\PP_j \in \cB^{m-1}(X)$, $j \ge 0$, and denote by $\PP_j[j]$ the binary complex that has the object~$\PP_j$ in degree $j$ and that is $0$ everywhere else. The binary complex $\PP_j[0]$ is similarly defined. By filtering~$\PP$ and using the shifting rule \cite[Lemma~1.6]{HKT17}, we obtain
\begin{align*}
[\PP &\otimes_{\Delta,m} \QQ; A_{m+1} \otimes_{\Delta,m} B_{m+1}, \ldots, A_n \otimes_{\Delta,m} B_n]\\
& = \sum_{j \ge 0} [\PP_j[j] \otimes_{\Delta,m}\QQ; A_{m+1,j}[j] \otimes_{\Delta,m} B_{m+1}, \ldots, A_{n,j}[j] \otimes_{\Delta,m} B_n]\\
& = \sum_{j \ge 0} (-1)^j[\PP_j[0] \otimes_{\Delta,m}\QQ; A_{m+1,j}[0] \otimes_{\Delta,m} B_{m+1}, \ldots, A_{n,j}[0] \otimes_{\Delta,m} B_n]\\
& = \sum_{j \ge 0} (-1)^j[\PP_j \otimes_{\Delta,m-1} \QQ; A_{m+1,j} \otimes_{\Delta,m-1} B_{m+1}, \ldots, A_{n,j} \otimes_{\Delta,m-1}B_n],
\end{align*}
where $\PP_j \otimes_{\Delta, m-1}\QQ $ means tensoring every object in $\QQ$ with $\PP_j$ and every differential in $\QQ$ with $1_{\PP_j}$. Note that this argument indeed applies to both the object $\PP \otimes_{\Delta,m}\QQ$ and to the automorphisms $A_{m+1} \otimes_{\Delta,m} B_{m+1}, \ldots, A_n \otimes_{\Delta,m} B_n$ as indicated in the chain of equalities above. The assumption that $\QQ$ is acyclic implies that $\PP_j[j] \otimes_{\Delta,m}\QQ$, $\PP_j[0] \otimes_{\Delta,m}\QQ$ and $\PP_j \otimes_{\Delta, m-1} \QQ$ are acyclic and hence all summands above are well-defined elements in $K_n(X)$. The assumption that $\PP$ is acyclic implies that the final alternating sum above vanishes, as was to be shown.
\end{proof}

\begin{rem}\label{rem: ArbitraryCubes}
If a multi-complex in $\cB^n(X)$ is supported on $[0,1]^n$, we call it an {\em $n$-cube}. Examples of $n$-cubes are the cubes of the form $[P;A_1, \ldots, A_n]$ considered above. In arbitrary $n$-cubes, none of the two isomorphisms in any binary isomorphism may be the identity, the binary isomorphisms in any one direction may not be equal to each other and the vertices may actually be different objects. Nonetheless, every $n$-cube $\PP$ is isomorphic to a cube of the form $[P; A_1, \ldots, A_n]$ as the following paragraph shows.

We write $\PP$ as a binary isomorphism
$\xymatrix{\PP_1 \binr{\sss \alpha}{\sss \beta} & \PP_0}$
between ($n-1$)-cubes $\PP_1, \PP_0 \in \cB^{n-1}(X)$. Then, the isomorphisms $\beta$ and $1$ provide an isomorphism between the $n$-cubes $\PP$ and $\xymatrix{\PP_0 \binr{\sss \alpha \beta^{-1}}{\sss 1} & \PP_0}$. By induction, we may assume that $\PP_0$ is of the form $[P, A_1, \ldots, A_{n-1}]$. The fact that the components of $\alpha \beta^{-1}$ commute with the identities in $\PP_0$ implies that all these components are equal to each other, say, to $A_n$. Hence, $\PP$ is isomorphic to $[P; A_1, \ldots, A_n]$, as claimed.

The previous paragraph implies that \cref{thm: Exterior Powres of Cubes} actually gives a formula for arbitrary $n$-cubes in $\cB^n(X)$. For notational reasons we have refrained from explaining the proof of \cref{thm: Exterior Powres of Cubes} in that generality. For future reference, we now at least describe the resulting formula.

Let $\PP$ be an cube in $\cB^n(X)$ with vertices $\PP_{(j_1, \ldots, j_n)}$, $(j_1, \ldots, j_n) \in \{0,1\}^n$. For each $i \in \{1, \ldots, r\}$, we define the $n$-cube $\mathrm{c}_i (\PP) \in \cB^n(X)$ as follows. Its vertex at $(j_1, \ldots, j_n) \in \{0,1\}^n$ is equal to $\cross_i \Lambda^r\left(P_{(j_1, \ldots, j_n)}, \ldots, P_{(j_1, \ldots, j_n)}\right)$ if $j_1 =1$ and equal to $\cross_i \Lambda^r\left(P_{(j_1, , \ldots, j_n)}, P_{(1, j_2, \ldots, j_n)}, \ldots, P_{(1, j_2, \ldots, j_n)}\right)$ if  $j_1 =0$.  For instance, if $n=3$, the $n$-cube $\mathrm{c}_i(\PP)$ looks as follows (writing $F_i$ for $\cross_i\Lambda^r$):
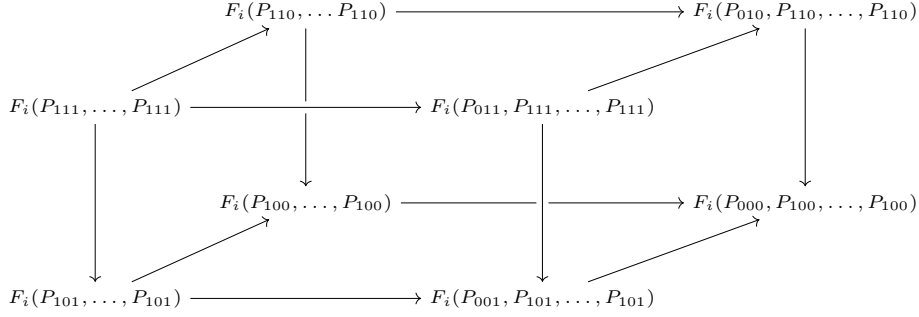
\begin{center} \tiny
\begin{tikzcd}[column sep={12em,between origins},row sep=3em,column sep=1em]
&
F_i(P_{110}, \ldots P_{110})
  \arrow[rr]
  \arrow[dd]
&&
F_i(P_{010}, P_{110}, \ldots, P_{110})
  \arrow[dd]
\\
F_i(P_{111}, \ldots, P_{111})
  \arrow[rr,crossing over]
  \arrow[ur]
  \arrow[dd]
&&
F_i(P_{011}, P_{111}, \ldots, P_{111})
  \arrow[ur]
\\
&
F_i(P_{100}, \ldots, P_{100})
  \arrow[rr]
&&
F_i(P_{000}, P_{100}, \ldots, P_{100})
\\
F_i(P_{101}, \ldots, P_{101})
  \arrow[rr]
  \arrow[ur]
&&
F_i(P_{001}, P_{101}, \ldots, P_{101})
  \arrow[from=uu,crossing over]
  \arrow[ur]
\end{tikzcd}
\end{center}
The edges in $\mathrm{c}_i(\PP)$ are binary isomorphisms of the form $\cross_i\Lambda^r(B_1, B_2, \ldots, B_n)$ where $B_i$ is the (binary) identity wherever this makes sense and is otherwise equal to the relevant binary isomorphism in the given cube~$\PP$. Then we have:
\[\lambda^r(\PP) = \sum_{i=1}^r (-1)^{i-1} [\mathrm{c}_i(\PP)] \quad \textrm{ in } \quad K_n(X).\]
\end{rem}

In \cite{Hiller}, Hiller has constructed exterior power operations on Quillen's $n^\mathrm{th}$ $K$-group $K_n(R)$ for any commutative ring $R$.

\begin{cor}\label{cor: Hiller}
Via the canonical isomorphism from Grayson's to Quillen's description of $K_1(R)$, the exterior power operations on $K_1(R)$ introduced in \cite{HKT17} agree with Hiller's.
\end{cor}

Assuming moreover that the external products introduced in \cref{sec: ExternalProducts} agree with those defined in, say, \cite[Definition IV.6.6]{K-book}, then \cref{cor: AdamsOperations} and \cite[Proposition~8.2]{Hiller} show that the respective Adams operations on $K_n(R)$ agree on external products of elements in $K_1(R)$. We of course expect that the exterior power operations defined in \cite{HKT17} agree with Hiller's on all of $K_n(R)$.

\begin{proof}
Let $K_0(\ZZ, R)$ denote the Grothendieck group of the exact category $\Aut(\cP(R))$. We know that $K_0(\ZZ,R)$ is in fact a commutative ring and is equipped with exterior power operations. Let $\widetilde{K}_0(\ZZ,R)$ denote the corresponding reduced Grothendieck group, i.e., the kernel of the forgetful map from $K_0(\ZZ,R)$ to $K_0(R)$. Let $D$ denote the epimorphism from $\widetilde{K}_0(\ZZ,R)$ to $K_1(R)$ obtained by restricting the natural epimorphism $K_0(\ZZ,R) \rightarrow K_1(R)$ to $\widetilde{K}_0(\ZZ,R)$.  By construction (see also the proof of \cite[Theorem~3.3]{AdamsOperations}), Hiller's exterior power operations on $K_1(R)$ are compatible, via $D$, with the exterior power operations on the $\lambda$-ideal $\widetilde{K}_0(\ZZ,R)$ of $K_0(\ZZ,R)$. Hence, we obtain for Hiller's operation $\lambda^r$ applied to the class $[P,A] \in K_1(R)$ of $(P,A) \in \Aut(\cP(R))$:
\begin{equation}\label{equ: HillerLambda}
\lambda^r ([P,A]) = \lambda^r(D([P,A]-[P,1]))=D(\lambda^r([P,A]-[P,1])).
\end{equation}
By a formula on page~2 of \cite{GraysonExterior}, we have:
\begin{align*}
\lambda^r&([P,A]-[P,1]) =\\
&\sum_{u=1}^r (-1)^u \sum_{}[(\Lambda^a(P)\otimes \Lambda^{b_1}(P) \otimes \cdots \otimes \Lambda^{b_u}(P), \Lambda^a(A) \otimes 1 \otimes \cdots \otimes 1]
\end{align*}
where the second sum is taken over all tuples $(a,b_1, \ldots, b_u)$ such that $a\ge 0$, $b_1 \ge 1, \ldots, b_u \ge 1$ and $a+b_1+\cdots + b_u =r$. The summands for $a=0$  vanish in $K_1(R)$; so, after re-indexing, the right-hand side of \cref{equ: HillerLambda} becomes
\begin{align*}
 \sum_{i=1}^{r}& (-1)^{i-1} \sum_{b_1+ \cdots +b_i=r}[(\Lambda^{b_1}(P) \otimes \cdots \otimes \Lambda^{b_i}(P), \Lambda^{b_1}(A) \otimes 1 \otimes \cdots \otimes 1]\\
&=\sum_{i=1}^r (-1)^{i-1} [\mathrm{cr}_i\Lambda^r(P, \ldots, P), \mathrm{cr}_i\Lambda^r(A, 1, \ldots, 1)].
\end{align*}
The latter expression is equal to the right hand-side of the formula in \cref{thm: Exterior Powres of Cubes} for $n=1$. Hence, \cref{cor: Hiller} follows from \cref{thm: Exterior Powres of Cubes}.
\end{proof}

\medskip

In \cite{Grech}, Grech introduces and studies modified Milnor $K$-groups $\widetilde{K}_n^\mathrm{M}(R)$, $n \ge 1$, for any commutative ring $R$. He shows that they agree with ordinary Milnor $K$-groups for instance when $R$ is a field, and he proves various properties for them, such as the resolution theorem and the cofinality theorem.  The group $\widetilde{K}_n^\mathrm{M}(R)$ is defined as the group generated by the isomorphism classes in $\Aut^n(\cP(R))$ modulo short-exact-sequence relations, multilinearity relations and Steinberg relations. By \cite[Theorem~5.3.4]{Grech}, mapping the class of $(P; A_1, \ldots, A_n) \in \Aut^n(\cP(R))$ in $\widetilde{K}_n^\mathrm{M}(R)$ to its class in $K_n(R)$ defines a natural homomorphism
\[\phi_n: \widetilde{K}_n^\mathrm{M}(R) \rightarrow K_n(R).\]

\begin{cor}\label{cor: MilnorInvariant}
The image of $\phi_n$ is invariant under $\lambda^r$ for all $r\ge 1$ and $n \ge 1$.
\end{cor}

\begin{proof}
 This immediately follows from Theorem~\ref{thm: Exterior Powres of Cubes}.
\end{proof}

\cref{cor: MilnorInvariant} raises the question whether there exist exterior power operations on $\widetilde{K}^\mathrm{M}_n(R)$ which are compatible, via $\phi_n$, with the exterior power operations on $K_n(R)$. More precisely, does the formula in \cref{thm: Exterior Powres of Cubes}, when considered as a formula in $\widetilde{K}^\mathrm{M}_n(R)$, define a well-defined operation on~$\widetilde{K}^\mathrm{M}_n(R)$? The  main difficulty here seems to be whether this formula factors modulo the Steinberg relations.

Operations on {\em classical} Milnor $K$-theory for fields (and certain smooth schemes) have been studied by Vial \cite{Vi}. He shows that all additive operations on $K_n^\mathrm{M}$ are in a sense trivial: more precisely, every endomorphism of the functor~$K^\mathrm{M}_n$, is multiplication by an integer. This is in agreement with the paragraph after \cref{thm: Exterior Powres of Cubes} above, which shows that exterior powers do indeed act by multiplication by an integer on the image of any classical Milnor $K$-theory element. We thank Oliver Br\"aunling for pointing us to \cite{Vi} and Charles Vial for elucidating emails.

\bibliographystyle{amsalpha}
\bibliography{computations}

\end{document}